\documentclass{ws-jaa}

\newtheorem{thm}{Theorem}
\newtheorem{lem}[thm]{Lemma}

\newtheorem{cor}[thm]{Corollary}

\newtheorem{rem}[thm]{Remark}
\newtheorem{ex}[thm]{Example}

\begin{document}

\markboth{Pinar Colak}
{Two-Sided Ideals in Leavitt Path Algebras}

\catchline{}{}{}{}{}
%
%%%%%%%%%%%%%%%%%%%%%%%%%%%%%%%%%%%%%%%%%%%%%%%%%%%%%%%%%%%%%%%%%%%%

\title{TWO-SIDED IDEALS IN LEAVITT PATH ALGEBRAS
}

\author{\footnotesize PINAR COLAK}

\address{Department Of Mathematics, Simon Fraser University\\ 8888 University Drive, Burnaby, BC
V5A 1S6, Canada\\
\email{ppekcagl@sfu.ca}}

\maketitle

\begin{history}
\received{(Day Month Year)}
\revised{(Day Month Year)}
\accepted{(Day Month Year)}
\comby{(xxxxxxxxx)}
\end{history}

\begin{abstract}
We explicitly describe two-sided ideals in
Leavitt path algebras associated with a row-finite graph. Our main result is that any two-sided ideal
$I$ of a Leavitt path algebra associated with a row-finite graph is generated by elements of the form
$v + \sum_{i=1}^n\lambda_i g^i$, where $g$ is a cycle based
at vertex $v$. We use this result to show that a Leavitt path
algebra is two-sided Noetherian if and only if the ascending
chain condition holds for hereditary and saturated closures of the
subsets of the vertices of the row-finite graph $E$.
\end{abstract}

\keywords{Leavitt path algebra; two-sided Noetherian; two-sided
ideal}

\ccode{2000 Mathematics Subject Classification: 16D70}

\vskip 5mm

Throughout this paper $K$ denotes a field. For a directed
graph $E$, the Leavitt path algebra $L_K(E)$ of $E$ with
coefficients in $K$ has received much recent attention, see e.g.
\cite{Abrams Pino}, \cite{Ara}, \cite{Goodearl}. The two-sided ideal
structure of $L_K(E)$ has been an important focus of much of this
work. In this paper we provide an explicit description of a
generating set for any two-sided ideal of $L_K(E)$, where $E$ is any
row-finite graph. We then use this description to identify those
row-finite graphs $E$ for which $L_K(E)$ is two-sided Noetherian.

We briefly recall the basic definitions.

A \emph{{directed graph}} $E= (E^0, E^1, r, s)$ consists of two
sets $E^0$, $E^1$ and functions $r,s: E^1 \rightarrow E^0.$ The
elements of $E^0$ are called \emph{{vertices}} and elements of $E^1$
are called \emph{{edges}}. For each $e \in E^1$, $r(e)$ is the
\emph{{range}} of $e$ and $s(e)$ is the \emph{{source}} of $e$. If
$r(e)=v$ and $s(e)=w$, then we say that $v$ \emph{{emits}} $e$ and
that $w$ \emph{{receives}} $e$. A vertex which emits no edges is
called a \emph{{sink}}. A graph is called \emph{{row-finite}} if
every vertex is the source of at most finitely many edges.

A \emph{{path}} $\mu$ in a graph $E$ is a sequence of edges
$\mu = e{_1} \cdots e{_n}$ such that $r(e{_i})=s(e{_{i+1}})$ for
$i=1, \dots ,n-1$. We define the source of $\mu$ by $s(\mu):=
s(e{_1})$ and the range of $\mu$ by $r(\mu):=r(e{_n})$. If we have
$r(\mu)=s(\mu)$ and $s(e{_i}) \neq s(e_{i+1})$ for every $i\neq j$,
then $\mu$ is called a \emph{{cycle}}. A \emph{{closed path based at}} $v$ is a path $\mu = e_1 \cdots e_n$,
with $e_j \in E^1$, $n \geq 1$ and such that $s(\mu) = r(\mu)= v$.
We denote the set of all such paths by $CP(v)$. A \emph{{closed
simple path based at}} $v$ is a closed path based at $v$, $\mu = e_1
\dots e_n$, such that $s(e_j) \neq v$ for $j > 1$. We denote the set
of all such paths by $CSP(v)$. Note that a cycle is a closed simple path based at any of its vertices. However the converse may not be true, as a closed simple path based at $v$ may visit some of its vertices (but not $v$) more than once.

Let $v$ be a vertex in $E^0$. If there is no cycle based at
$v$, then we let $g=v$ and call it the \emph{trivial} cycle. If $g$
is a cycle based at $v$ of length at least 1, then $g$ is called a
\emph{non-trivial} cycle.

Let $E = (E^0, E^1)$ be any directed graph, and let $K$ be a field. We define the \emph{Leavitt path $K$-algebra} $L_K(E)$ associated with $E$ as the $K$-algebra generated by a set ${v \in E^0}$ together with a set $\{e, e^* | e \in E^1\}$, which satisfy the following relations:
\begin{enumerate}
\item $vv' = \delta_{v,v'} v$ for all $v, v' \in E^0$.
\item $s(e)e=er(e)=e$ for all $e\in E^1$.
\item $r(e)e^* = e^*s(e)=e^*$ for all $e \in E^1$.
\item $e^* f = \delta_{e,f}r(e)$ for all $e, f \in E^1$.
\item $v = \sum_{\{e \in E^1 | s(e)=v \}} ee^*$ for every $v \in E^0$ such that $0< |s^{-1}(v)|< \infty$.
\end{enumerate}

The elements of $E^1$ are called \emph{real edges}, while for $e \in E^1$ we call $e^*$ a \emph{ghost
edge}. The set $\{e^* | e \in E^1 \}$ is denoted by $(E^1)^*$. We let $r(e^*)$ denote $s(e)$,
and we let $s(e^*)$ denote $r(e)$. We say that a path in $L_K(E)$ is a \emph{{real path}}
(resp., a \emph{{ghost path}}) if it contains no terms of the
form $e^*_i$ (resp., $e_i$). We say that $p \in L_K(E)$ is a
polynomial in \emph{{only real edges}} (resp., in \emph{{only
ghost edges}}) if it is a sum of real paths (resp., ghost
edges). The \emph{length} of a real path (resp., ghost path)
$\mu$, denoted by $|\mu|$, is the number of edges it contains. The
length of $v \in E^0$ is $0$. Let $x$ be a polynomial in only real
edges (resp., in only ghost edges) in $L_K(E)$. If $x = \mu_1
+ \cdots + \mu_n$, where the $\mu_i$'s are real paths (resp.,
ghost paths), then the \emph{length} of $x$, denoted by $|x|$, is
defined as $\mathrm{max}_{i=1,\dots,n} \{|\mu_i|\}$. An edge $e$ is
called an \emph{{exit}} to the path $\mu = e_1 \cdots e_n$ if there
exists $i$ such that $s(e)=s(e_i)$ and $e \neq e_i$.

The proofs of the following can be found in \cite{Abrams Pino}.

\begin{lem}
$L_K(E)$ is spanned as a $K$-vector space by monomials
\begin{enumerate}
\item $kv_i$ with $k \in K$ and $v_i \in E^0$, or
\item $ke_{1} \dots e_{a}f^*_{1} \dots f^*_{b}$ where $k \in K$; $a, b \geq 0$, $a + b > 0$, $e_{1}, \dots, e_a,f_1, \dots, f_b \in E^1$.
\end{enumerate}
\end{lem}

\begin{lem}
If $\mu, \nu \in CSP(v)$, then $\mu^*\nu=\delta_{\mu,\nu}v$. For
every $\mu \in CP(v)$ there exist unique $\mu_1, \dots, \mu_m \in
CSP(v)$ such that $\mu=\mu_1 \cdots \mu_m$.
\end{lem}

For a given graph $E$ we define a preorder $\geq$ on the vertex set
$E^0$ by: $v \geq w$ if and only if $v=w$ or there is a path $\mu$
such that $s(\mu )=v$ and $r(\mu )=w$. We say that a subset $H \subseteq E^0$ is \emph{{hereditary}}
if $w \in H$ and $w \geq v$ imply $v \in H$. We say $H$ is
\emph{{saturated}} if whenever $0<|s^{-1}(v)|<\infty$ and $\{
r(e) : s(e)=v \} \subseteq H$, then $v \in H$. The \emph{hereditary saturated closure} of a set $X \subset E^0$
is defined as the smallest hereditary and saturated subset of $E^0$
containing $X$. For the hereditary saturated closure of $X$ we use the notation given in % (Abrams and Aranda Pino, 2008)
\cite{Abrams Pino arbitrary}: $\overline{X}= \bigcup^{\infty}_{n=0}
\Lambda_n(X)$, where
$$\Lambda_0(X):= \{v \in E^0 ~|~ x \geq v ~ \mathrm{for ~ some} ~ x \in X\}, {\rm ~and~for}~n\geq 1, $$ $$\Lambda_n(X):= \{ y \in E^0 ~|~ 0 < |s^{-1}(y)| < \infty ~ \mathrm{and} ~ r(s^{-1}(y)) \subseteq \Lambda_{n-1} (X) \} \cup \Lambda_{n-1}(X).$$

\begin{ex} \label{ex:5}
Let $E=(E^0, E^1, r, s)$ be the directed graph where $E^0=\{v,w\}$ and
$E^1=\{e_1, e_2, e_3\}$ such that $r(e_1)=s(e_1)=v$ and
$r(e_2)=r(e_3)=s(e_3)=w$.

Then $\overline{\{v_1\}}=\{v_1, v_2\}$, whereas $\overline{ \{v_2 \}}=\{v_2\}$.
\end{ex}

\begin{ex} \label{ex:4}
Let $E=(E^0, E^1, r, s)$ be a directed graph where $E^0=\{v_i ~|~i
\in \mathbb{Z}\}$ and $E^1=\{e_i~|~i \in \mathbb{Z} \}$ such that
$r(e_i)=v_i$ and $s(e_i)=v_{i-1}$.

Let $X=\{v_0\}$. Then we get $\Lambda_0 (X) = \{v_0, v_1, \dots\}$. Furthermore,
\begin{align*}
\Lambda_1(X) & =\Lambda_0\{v_0\} \cup \{ y \in E^0 ~|~ 0 < |s^{-1}(y)| < \infty ~ \mathrm{and} ~ r(s^{-1}(y)) \subseteq \Lambda_{0} (X) \} \\
             & = \{v_0, v_1, \dots\} \cup \{v_{-1}\} \\
             & = \{v_{-1}, v_0, v_1, \dots\}.
\end{align*}

Similarly, $\Lambda_k(X)=\{v_{-k}, v_{-k+1}, \dots\}$, and hence $\overline{X}= \bigcup^{\infty}_{n=0} \Lambda_n(X)=E^0$.
\end{ex}

With the introductory remarks now complete, we begin our discussion of the main result with the following important observation.

\begin{rem} \label{rem:source-range}
If $I$ is a two-sided ideal of $L_K(E)$ and $\mu=\mu_1+\cdots+\mu_n$ is in
$I$, where $\mu_1, \dots, \mu_n$ are monomials in $L_K(E)$, then
$\gamma_i=s(\mu_i)\mu r(\mu_i)$ is the sum of those $\mu_j$ whose
sources are all the same and whose ranges are all the same;
specifically, the sum of those $\mu_j$ for which $s(\mu_j)=s(\mu_i)$
and $r(\mu_j)=r(\mu_i)$. Moreover, $\gamma_i \in I$. Thus we may
write $\mu = \gamma_1 + \cdots + \gamma_m$, with each $\gamma_i$
with the above properties.
\end{rem}

\noindent {\bf Notation.} Let $L_K(E)_{\rm{R}}$ (resp., $L_K(E)_{\rm{G}}$) be the
subring of elements in $L_K(E)$ whose terms involve only real edges
(resp., only ghost edges). \vskip 2mm

\begin{lem} \label{lem:real} Let $I$ be a two-sided ideal
of $L_K(E)$ and $I_{\rm{real}}=I \cap L_K(E)_{\rm{R}}$. Then $I_{\rm real}$ is the
two-sided ideal of $L_K(E)_{\rm R}$ generated by elements of $I_{\rm{real}}$ having the form $v+ \sum^n_{i=1} \lambda_i g^i$,
where $v \in E^0$, $g$ is a cycle based at $v$ and $\lambda_i \in K$ for $1 \leq i \leq n$.
\end{lem}

\begin{proof} Let $J$ be the ideal of $L_K(E)_{\rm R}$
generated by elements in $I_{\rm{real}}$ of the indicated form. Our
claim is $J=I_{\rm{real}}$. Towards a contradiction, suppose
$I_{\rm{real}} \setminus J \neq \emptyset$; choose $\mu \in
I_{\rm{real}} \setminus J$ of minimal length. By Remark
\ref{rem:source-range}, we can write $\mu = \tau_1 + \cdots + \tau_m$
with each $\tau_i$ is in $I_{\rm{real}}$ and is the sum of those
paths whose sources are all the same and whose ranges are all the
same. Since $\mu \not \in J$, one of the $\tau_i \not \in J$.
Replacing $\mu$ by $\tau_i$, we may assume that $\mu = \lambda_1\mu_1
+ \cdots + \lambda_n\mu_n$ where all the $\mu_i$ have the same
source and the same range. First we claim that one of the $\mu_i$
must have length 0, i.e. $\mu_i = kv$ for some vertex $v \in E^0$
and $k \in K$. Suppose not. Then for each $i$ we can write
$\mu_i=e_i\nu_i$ where $e_i \in E^1$. So $\mu= \sum^n_{i=1}
e_i\nu_i$. Now
$$e_i^*\mu=\sum_{\{j| e_j=e_i\}} \lambda_j \nu_j \in I_{\rm{real}}$$ and has smaller
length than $\mu$. So $e_i^*\mu \in J$ and hence clearly $e_ie_i^*\mu \in J$.
Then $$\mu = \sum_{{\rm distinct}~ e_i} e_ie_i^*\mu \in J,$$ a
contradiction. So we can assume without loss of generality that
$\mu_1=kv$, with $v$ a vertex. Since all the terms in $\mu$ have the same source
and the same range, each $\mu_i$ is a closed path based
at $v$. Multiplying by a scalar if necessary we can write $\mu = v +
\lambda_2\mu_2 + \cdots + \lambda_n\mu_n$, $\mu_i$ closed paths at
$v$.

\textbf{Case I:} There exists no, or exactly one, closed simple path at $v$. If there are no closed simple paths at $v$ then we get $\mu \in J$, a contradiction. If there is exactly one closed simple path $g$ based at $v$ then necessarily $g$ must be a cycle. Furthermore, the only paths in $E$ which have source and range equal to $v$ are powers of $g$. Then $\mu = v + \sum_{i=2}^n \lambda_i g^{m_i} \in J$, a contradiction.

\textbf{Case II:} There exist at least two distinct closed simple
paths $g_1$ and $g_2$ based at $v$. As $g_1 \neq g_2$ and neither is
a subpath of the other, $g_2^*g_1=0=g_1^*g_2$. Without loss of
generality assume $|\mu_2| \geq \cdots \geq |\mu_n| \geq 1$. Then
for some $k \in \mathbb{N}$ $|g_1^k|>|\mu_2|$. Multiplying by
$(g_1^*)^k$ on the left and $g_1^k$ on the right, we get
$$\mu'=(g^*_1)^k \mu(g_1)^k=v + \sum_{i=2}^n \lambda_i(g_1^*)^k \mu_i (g_1)^k.$$ Note that if $0 \neq (g_1^*)^k \mu_i (g_1)^k$, then $(g_1^*)^k \mu_i \neq 0$. Since $|g_1^k|>|\mu_i|$, we get $g_1^k= \mu_i\mu_i'$ for some path $\mu_i'$. Since the $\mu_i$ are closed paths based at the vertex $v$, one gets from the equation $(g_1)^k = \mu_i\mu_i'$ that $\mu_i = (g_1)^r$ for some integer $r \leq k$. So $\mu_i$ commutes with $(g_1)^k$ and thus each non-zero term $(g_1^*)^k\mu_i(g_1)^k = \mu_i$.

%Note that if $0 \neq (g_1^*)^k \mu_i (g_1)^k$, then $(g_1^*)^k \mu_i \neq 0$. Since $|g_1^k|>|\mu_i|$, we get $g_1^k= \mu_i\mu_i'$ for some path $\mu_i'$ for every $i \in \{2, \dots, n \}$. Then $(g^*_1)^k \mu_i(g_1)^k={\mu_i'}^*\mu_i\mu_i' \neq 0$, and since $|\mu_i\mu_i'|> |\mu_i'|$, we get ${\mu_i'}^*\mu_i\mu_i' = \nu_i$ for some real path $\nu_i$.

%So $$\mu'= v + \lambda_2\nu_2 + \cdots + \lambda_n\nu_n$$ where $\nu_i$ is a part of the path $g_1^k$.
Since $g_2^*g_1=0$, $g_2^*\mu_i = 0$ for every $i \in \{2, \dots, n \}$
and so we get $g_2^*\mu'g_2 = g_2^*vg_2 = v \in I \cap L_K(E)_{\rm R}= I_{\rm real}$,
which implies that $v$ is in $J$. Then $\mu = \mu v \in J$, a contradiction.
\end{proof}

%\textbf{Case II:} There are at least two closed simple paths at $v$. Let's call these paths $g_1$ and $g_2$. Note that $g^*_1g_2=g^*_2g_1=0$. Without loss of generality assume $|\mu_2| \geq \cdots  \geq |\mu_n|$. There exists $k \in \mathbb{N}$ such that $|g^k_1| > |\mu_2|$. Multiply $\mu$ by $(g^*_1)^k$ on the left and by $g^k_1$ on the right. Then we get $$\mu'=(g^*_1)^k \mu(g_1)^k = (g^*_1)^k v (g_1)^k + (g^*_1)^k {\mu}_2(g_1)^k + \dots + (g^*_1)^k{\mu}_{n}(g_1)^k = v+ {\gamma_1} {\mu'}_1 + \dots + {\gamma_m}{\mu'}_m,$$ where each ${\mu'}_i$ is $g_1^{m_i}$ for some $m_i \in \mathbb{N}$ and $\gamma_i$ is in $K$ for $i=1,\dots,m$. Next we multiply $\mu'$ by $g^*_2$ on the left and by $g_2$ on the right. Since $g_2^* g_1=0$, we get $g^*_2 \mu' g_2 = g^*_2 v g_2 = v$. But this means $v$ is in $I_{\rm{real}}$ and hence is in $J$. Then note that $\mu=v\mu$ is also in $J$, a contradiction.

%We conclude $I_{\rm{real}}$ is generated by elements of the form $v+ \sum^m_{k=1} \lambda_kg^k$, where $v \in E^0$, $g$ is a cycle at $v$ and $\lambda_1, \dots, \lambda_m \in K$. \qed \\

It can be easily shown that the analogue of Lemma \ref{lem:real} is true for $I_{\mathrm{ghost}}$. We state this for the sake of completeness.

\begin{lem}
Let $I$ be a two-sided ideal of $L_K(E)$. Then $I_{\mathrm{ghost}}$ is
generated by elements of the form $v+ \sum^m_{k=1}
\lambda_k(g^*)^k$, where $v \in E^0$, $g$ is a cycle at $v$ and
$\lambda_i \in K$ for $1 \leq i \leq n$.
\end{lem}

\begin{thm} \label{THM}
Let $E$ be a row-finite graph. Let $I$ be any two-sided ideal of
$L_K(E)$. Then $I$ is generated by elements of the form $v+
\sum^m_{k=1} \lambda_kg^k$, where $v \in E^0$, $g$ is a cycle at $v$
and $\lambda_1, \dots, \lambda_m \in K$.
\end{thm}

\begin{proof} Let $J$ be the two-sided ideal of $L_K(E)$
generated by $I_{\rm{real}}$. By Lemma \ref{lem:real}, it is enough if we show that $I=J$.
Suppose not. Choose $x = \sum^d_{i=1} \lambda_i\mu_i\nu_i^*$ in $I
\setminus J$, where $d$ is minimal and $\mu_1, \dots, \mu_d, \nu_1,
\dots, \nu_d$ are real paths in $L_K(E)_R$. By Remark
\ref{rem:source-range}, $x = \alpha_1+ \cdots + \alpha_m$, where
each $\alpha_i \in I$ and is a sum of those monomials all having the
same source and same range. Since $x\not \in J$, $\alpha_j \not \in
J$ for some $j$. By the minimality of $d$, we can replace $x$ by
$\alpha_i$. Thus we we can assume that $x=\sum^d_{i=1}
\lambda_i\mu_i\nu_i^*$, where for all $i$, $j$, $s(\mu_i\nu_i^*)=
s(\mu_j\nu_j^*)$ and $r(\mu_i\nu_i^*)=r(\mu_j\nu_j^*)$. Among all
such $x =\sum^d_{i=1} \lambda_i\mu_i\nu_i^* \in I \setminus J$ with
minimal $d$, select one for which $(|\nu_1|, \dots, |\nu_d|)$ is the
smallest in the lexicographic order of $(\mathbb{Z}^+)^d$. First
note that we have $|\nu_i| > 0$ for some $i$, otherwise $x$ is in
$I_{\rm real} \subset J.$ Let $e$ be in $E^1$. Then note that $$xe =
\sum^d_{i=1} \lambda_i\mu_i\nu_i^*e=
\sum^{d'}_{i=1}\lambda_i\mu_i(\nu_i')^*$$ either has fewer terms
$(d' < d)$, or $d= d'$ and $(|\nu_i'|, \dots, |\nu_d'|)$ is smaller
than $(|\nu_1|, \dots, |\nu_d|)$. Then by minimality, we get $xe$ is
in $J$ for every $e\in E^1$. Since $|\nu_i| >0$ for some $i$, $w$ is
not a sink and emits finitely many edges. Hence we have $$x = xw = x
\sum_{\{e_j \in E^1 : s(e_j)=v\}} e_je_j^*= \sum_{\{e_j \in E^1 :
s(e_j)=v\}} (xe_j)e_j^* \in J.$$ We get a contradiction, so the
result follows.
\end{proof}

%Ranga's proof:
%We claim that $|\nu_i|=0$ for some $i$. Suppose, on the contrary,
%$|\nu_i| > 0$ for all $i$, so that $\nu_i = e_i\nu_i'$ with $e_i \in
%E^1$, for all $i = 1, \dots, d$. Then $$xe = \sum^d_{i=1}
%\lambda_i\mu_i\nu_i^*e= \sum^d_{\{j| e_j=e\}}
%\lambda_j\mu_j(\nu_j')^* \in I,$$ for any $e\in E^0.$ If the number
%of monomial terms in $xe$ is less than $d$, then $x \in J$. If the
%number of monomial terms in $xe$ is $d$, then since $(|\nu_1'|,
%\dots, |\nu'_d|) < (|\nu_1|, \dots, |\nu_d|)$, the minimality
%condition implies $xe \in J$. Clearly, $xee^* \in J$ and $$xee^*=
%\sum^d_{\{j| e_j=e\}} \lambda_j\mu_j(\nu_j')^*e^*= \sum^d_{\{j|
%e_j=e\}} \lambda_j\mu_j\nu_j^* \in J.$$ Then $$x = \sum_{\{{\rm
%distinct}~ e_j | j\in \{1, \dots, d\}\}} xe_je_j^* \in J,$$ a
%contradiction. So one of the $\nu_j$, say $\nu_1$, is of the form
%$kv$ for some $v \in E^0$ and $0 \neq k \in K$. Multiplying by
%$k^{-1}$ if necessary and replacing $k^{-1}x$ by $x$ we can write $x
%=\mu_1v+ \lambda_2\mu_2\nu_2^*+ \cdots + \lambda_d\mu_d\nu_d^*$,
%where the $\mu_i$ have the same source $v$ and sam range $v$. Now
%$\mu_1v=\mu_1 \in I_{\rm real} \subset J$. Also by the minimality of
%$d$, $\lambda_2\mu_2\nu_2^*+ \cdots + \lambda_d\mu_d\nu_d^* \in J$.
%Hence $x\in J$, a contradiction. Thus $I = J$. \qed

\begin{rem}
We note that the Theorem does not hold for arbitrary graphs. An example is the ``infinite clock'': Let $E^0 = \{v, w_1, w_2, \dots\}$ and $E^1=\{e_1,e_2, \dots \}$ with $r(e_i)=w_i$ and $s(e_i)=v$. Then the two-sided ideal generated by $v - e_1e_1^*$ is not generated by the elements of the desired form.
\end{rem}

\noindent {\bf Notation.} The element $v + \sum_{i=1}^n \lambda_ig^i$ is denoted by $p(g)$, where $p(x) = 1 + \lambda_2x+\cdots+ \lambda_nx^n \in K[x]$. \vskip 2mm

%\begin{lem} \label{lem:nontrivial}
%Suppose $g_1$ and $g_2$ are two different non-trivial cycles based
%at the vertex $v \in E^0$. Let $p(g_1)=v +
%\sum_{i=1}^n\lambda_ig_1^i$. Then $I=<p(g_1)>=<v>$.
%\end{lem}

%\noindent {\bf Proof.} It is sufficient to show that $v$ is in $I$.
%We have $p(g_1)=v + \sum_{i=1}^n\lambda_ig_1^i$, then
%$g_2^*p(g_1)g_2 \in I$ and $$g_2^*p(g_1)g_2 = g_2^*vg_2 +
%g_2^*\sum_{i=1}^n\lambda_ig_1^i g_2 = v \in I.$$ \qed

Now the Theorem is in hand, we are going to put the pieces together
to get the Noetherian result.

\begin{rem} \label{rem:euclidean}
Let $g$ be a cycle based at $v \in E^0$ and let $p_1(x), p_2(x) \in K[x]$ be such that $p_1(g), p_2(g) \in I$. If we let $q(x)={\rm gcd}(p_1(x),p_2(x)) \in K[x]$, then $q(g) \in <p_1(g), p_2(g)>$.
\end{rem}

\begin{lem} \label{lem: Ranga 2}
Let $I$ be a two-sided ideal of $L_K(E)$, where $E$ is an arbitrary
graph. Suppose $g$, $h$ are two non-trivial cycles based at distinct vertices
$u$, $v$ respectively. Suppose $u+ \sum a_rg^r = p(g)$ and $v+ \sum
b_sh^s = q(h)$ both belong to $I$, where $p(x)$ and $q(x)$ are
polynomials of smallest positive degree in $K[x]$ with $p(0) = 1 =
q(0)$ such that $p(g) \in I$ and $q(h) \in I$. If $u \geq v$, then
$q(h) \in <p(g)>.$
\end{lem}

\begin{proof} Let $\mu$ be a path from $u$ to $v$. We claim $v$
must lie on the cycle $g$. Because, otherwise, $\mu^*g = 0$ and so
$\mu^*p(g)\mu = \mu^*u\mu + \sum a_r\mu^*g^r\mu = v \in I$. This
contradicts the fact that ${\rm deg}q(x) >0$. So we can write $g=
\mu\nu$ and $h=\nu\mu$ where $\nu$ is the part of $g$ from $v$ to
$u$. Since $\mu^*g\mu = h$, we get $\mu^*p(g)\mu = p(h) \in I$. By
the minimality of $q(x)$, $q(x)$ is a divisor of $p(x)$ in $K[x]$.
Similarly, since $\nu^* q(h)\nu = q(g)\in I$, we conclude that
$p(x)$ is a divisor of $q(x)$. Thus $q(x) = kp(x)$ for some $k \in
K$. Since $p(0)=1=q(0)$, $q(x)=p(x)$. Hence $q(h) = p(h) = \mu^*
p(g) \mu \in <p(g)>$.
\end{proof}

The next Lemma and its proof is implicit in the proof of Lemma 7 in
\cite{Abrams 2006}.
%(Abrams and Aranda Pino, 2006).
\begin{lem} \label{lem: Ranga 3}
Let $E$ be an arbitrary graph and $S \subset E^0$. If $v \in
\overline{S}$, and there is a non-trivial cycle based at $v$, then
$u \geq v$ for some $u \in S$.
\end{lem}

\begin{proof} We recall that  $\overline{S} = \cup_{n \geq 0}
\Lambda_n(S)$. Let $k$ be the smallest non-negative integer such
that $v \in \Lambda_k(S)$. We prove the Lemma by the induction on
$k$, the Lemma being true by definition when $k=0$. Assume $k >0$ and that the
Lemma holds when $k=n-1$. Let $k=n$. Since $v \in \Lambda_n(S)\setminus
\Lambda_{n-1}(S)$, $0 < |s^{-1}(v)| < \infty$ and $\{w_1, \dots,
w_m\} = r(s^{-1}(v)) \subset \Lambda_{n-1}(S)$. Since $v$ is the
base of a non-trivial cycle $g$, one of the vertices, say, $w_j$
lies on the cycle $g$ and so $w_j \geq v$. Since $w_j \in
\Lambda_{n-1}(S)$ and is the base of a cycle, by induction there is a $u \in S$
such that $u \geq w_j$. Then $u \geq v$, as desired.
\end{proof}

We also need the following Lemma, whose proof is given in the first
paragraph of the proof of Theorem 5.7 in \cite{Tomforde}. % (Tomforde, 2007).

\begin{lem} \label{lem: Ranga 4}
Let $E$ be an arbitrary graph and let $H$ be a hereditary and
saturated subset of vertices in $E$. If $I$ is the two-sided ideal
generated by $H$, then $I \cap E^0= H$.
\end{lem}

\begin{thm} \label{Thm: Noetherian}
%Let $E$ be a row-finite arbitrary graph. Then $L_K(E)$ is two-sided
%Noetherian if and only if the hereditary and saturated subsets of the vertices in
%$E^0$ satisfy the ascending chain condition (equivalently, $L_K(E)$ is graded two-sided Noetherian).
Let $E$ be a row-finite graph. Then the following are equivalent:

\begin{arabiclist}
\item $L_K(E)$ has a.c.c. on two-sided ideals,
\item $L_K(E)$ has a.c.c. on two-sided graded ideals,
\item The hereditary saturated closures of the subsets of the vertices in $E^0$
satisfy a.c.c..
\end{arabiclist}
\end{thm}

\begin{proof} $(3) \Rightarrow (1)$ Suppose the ascending
chain condition holds on the hereditary saturated closures of the subsets of
$E^0$. Let $I$ be a two-sided ideal of $L_K(E)$. By Theorem
\ref{THM} and by Remark \ref{rem:euclidean}, $I$ is generated by the set
\begin{eqnarray*}
T = &\{v+ \sum_r\lambda_rg^r = p(g) \in I ~|~ v \in E^0 , g {\rm ~is
~a ~cycle ~(may~be~trivial)~based~at~}v~{\rm and} \\
&~p(x) \in K[x]{\rm
~is~a~polynomial~of~smallest~degree~such~that~}p(g)\in
I{\rm~and~}p(0)=1\}.
\end{eqnarray*}

It is well known that two-sided Noetherian is equivalent to every
two-sided ideal being finitely generated, so we wish to show that
$I$ is generated by a finite subset of $T$.

Suppose, towards a contradiction, there are infinitely many
$p_i(g_i) = v_i + \sum\lambda_rg_i^r \in T$ with $i \in H$, an
infinite set and for each $i$, $g_i$ is a non-trivial cycle based at
$v_i$ and that ${\rm deg} p_i(x) >0$. By Lemma \ref{lem: Ranga 2},
we may assume that for any two $i$, $j$ with $i \neq j$, $v_i \ngeq
v_j$. Well-order the set $H$ and consider it as the set of all
ordinals less than an infinite ordinal $\kappa$. Define $S_1 =
{v_1}$ and for any $\alpha < \kappa$, define $S_\alpha =
\cup_{\beta<\alpha}S_\beta$ if $\alpha$ is a limit ordinal, and
define $S_\alpha=S_\beta \cup \{v_{\beta+1}\}$ if $\alpha$ is a
non-limit ordinal of the form $\beta+1$. By the hypothesis the
ascending chain of hereditary saturated closures of subsets $\overline{S}_1
\subset \overline{S}_2 \subset \cdots \overline{S}_\alpha \subset
\cdots$ becomes stationary after a finite number of steps. So there
is an integer $n$ such that $\overline{S}_n = \overline{S}_{n+1} =
\cdots$. Now $v_{n+1} \in \overline{S}_{n+1} = \overline{S}_n$ and
by Lemma \ref{lem: Ranga 3}, there is a $v_i \in S_n$ such that $v_i
\geq v_{n+1}$. This is a contradiction. Hence the set $W =
\{p_i(g_i) \in T ~|~ {\rm deg}p_i(x) >0 \}$ is finite.

So by the previous paragraph, if there are only finitely many
$p_i(g_i)$ in $T$ with ${\rm deg}p_i(x)=0$, that is, only finitely
many vertices in $T$, then we are done. We index the vertices
$v_\alpha$ in $T$ by ordinals $\alpha < \kappa$, an infinite
ordinal, then as before, we get a well-ordered ascending chain of
hereditary saturated closure of subsets $\overline{S}_1 \subset
\overline{S}_2 \subset \cdots \subset \overline{S}_\alpha \subset
\cdots$ $(\alpha < \kappa)$ where $S_1 = \{v_1\}$ and the $S_\alpha$
are inductively defined as before. Since, by hypothesis, this chain
becomes stationary after a finite number of steps, there is an
integer $n$ such that $\overline{S}_\alpha = \overline{S}_n$ for all
$\alpha > n$. Thus $\{ v_\alpha ~|~\alpha < \kappa\} \subset
\overline{S}_n$. Since the ideal generated by the finite set $S_n =
\{v_1, \dots, v_n\}$ contains $\overline{S}_n$, we conclude that the
ideal $I$ is generated by the finite set $W \cup S_n$. Thus the
Leavitt path algebra is two-sided Noetherian.

$(1) \Rightarrow (3)$ Conversely, suppose $L_K(E)$ is two-sided
Noetherian. Consider an ascending chain of hereditary saturated
closures of subsets of vertices $\overline{S}_1 \subset
\overline{S}_2 \subset \cdots$ in $E^0$. Consider the corresponding
ascending chain of two-sided ideals $I_1 \subset I_2 \subset
\cdots,$ where for each integer $i$, $I_i$ is the two-sided ideal
generated by $\overline{S}_i$. By hypothesis, there is an integer
$n$ such that $I_n=I_i$ for all $i
>n$. We claim that $\overline{S}_i = \overline{S}_n$ for all $i>n$.
Otherwise, we can find a vertex $w \in \overline{S}_i \setminus
\overline{S}_n$ and since $w \in I_i=I_n$, $w \in I_n \cap E^0 =
\overline{S}_n$ by Lemma \ref{lem: Ranga 4} and this is a
contradiction.

$(1) \Leftrightarrow (2)$ It is well-known (see \cite{Tomforde})
that if $I(H)$ is a two-sided ideal of $L_K(E)$ generated by a
hereditary and saturated subset $H$ of $E^0$, then $I(H)$ is a
graded ideal of $L_K(E)$. If we call $L_K(E)$ graded two-sided
Noetherian if graded two-sided ideals of $L_K(E)$ satisfy the
ascending chain condition, then Theorem \ref{Thm: Noetherian} states
that for any graph $E$, $L_K(E)$ is two-sided Noetherian if and only
if it is graded two-sided Noetherian.
\end{proof}

\begin{rem}
We note that this result only shows that the a.c.c. on graded ideals
is sufficient to get a.c.c. on all ideals, and that we are not
proving that every ideal in a Noetherian Leavitt path algebra is
graded. As an example we can consider $K[x, x^{-1}]$, which is the Leavitt path algebra of the
graph with one vertex and one loop. Note that although this Leavitt path algebra has infinitely many ideals, it is
nonetheless Noetherian, but has only the trivial graded ideals.
\end{rem}

Now we easily see
\begin{cor}
Every Leavitt path algebra with a finite graph is two-sided Noetherian.
\end{cor}

We conclude by presenting another example of a non-Noetherian Leavitt path algebra.

\begin{ex} \label{ex:2}
Let $E=(E^0, E^1, r, s)$ be the directed graph where
$E^0=\{v,w_1,w_2,w_3, \dots \}$ and $E^1=\{e_1, e_2, \dots \}\cup
\{f_1,f_2, \dots\}$ is such that $r(e_i)=v$ and
$s(e_i)=r(f_i)=s(f_i)=w_i$. The graph of this Leavitt path algebra
is given in Figure \ref{figure:2}.
\begin{figure} \label{figure:2}
\centerline{\psfig{file=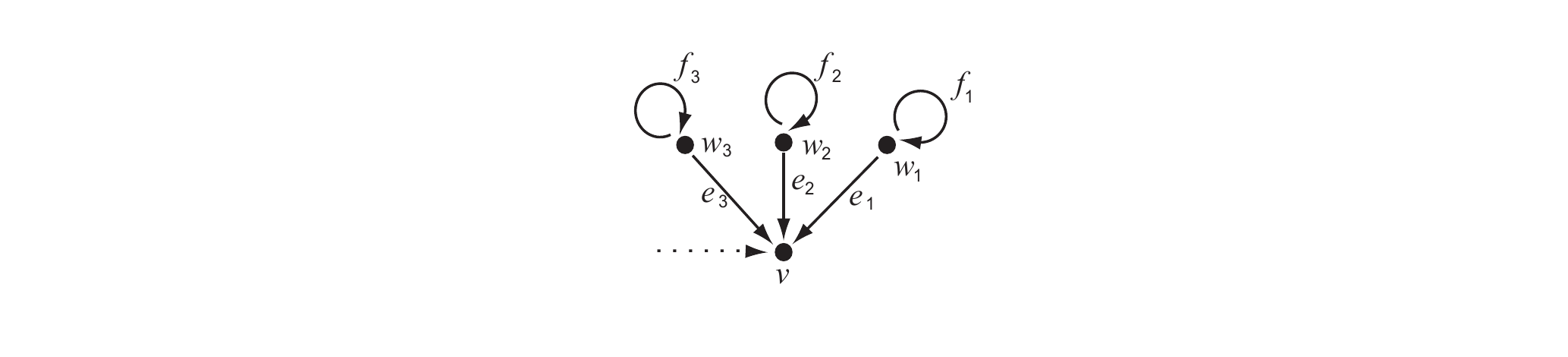}}
\vspace*{8pt}
\caption{Leavitt path algebra defined in Example \ref{ex:2}.}
\end{figure}
Note that if we let $S_i=\{w_1, \dots, w_i\}$, then $\overline{S_1}
\subset \overline{S_2} \subset \cdots$ is a non-terminating
ascending chain of hereditary saturated closures of sets in
$E^0$. Hence by Theorem \ref{THM}, $L_K(E)$ is not two-sided
Noetherian. Indeed, $<w_1> \subset <w_1, w_2> \subset \cdots$ is a
non-terminating ascending chain of ideals in $L_K(E)$.
\end{ex}

In \cite{ABCR} we present
some additional consequences of Theorem \ref{THM}, including a description of the
two-sided artinian Leavitt path algebras.

\section*{Acknowledgments}
The author is indebted to Dr. Jason P. Bell for the guidance and the technical discussions since the inception of this work, as well as to Dr. Gene Abrams and to Dr. Kulumani M. Rangaswamy for their thorough reviews and support. The author also thanks Dr. Gonzalo Aranda Pino and Dr. Kathi Crow for their valuable feedback on the first draft of this paper, and Azhvan Sheikh Ahmady for reading and improving this manuscript.

\end{document}